\documentclass[11pt]{amsart}
\title[Categories of  complex variations of Hodge structure]{Categories of  complex variations of Hodge structure over compact K\"ahler manifolds}

\author{Hisashi Kasuya}

\usepackage{fullpage}
\usepackage{amssymb}
\usepackage{amsmath}
\usepackage{amscd}
\usepackage{amstext}
\usepackage{amsfonts}
\usepackage[all]{xy}

\theoremstyle{plain}

\theoremstyle{plain}

\theoremstyle{plain}

\theoremstyle{plain}

\theoremstyle{plain}
\newtheorem{thm}{Theorem}[subsection] 
\theoremstyle{remark}
\newtheorem{remark}[thm]{Remark}
\theoremstyle{remark}
\newtheorem{Important note}[thm]{Important note}
\theoremstyle{Main result}
\newtheorem{main result}{Main result}
\theoremstyle{lemma}
\newtheorem{lemma}[thm]{Lemma}
\theoremstyle{definition}

\theoremstyle{proposition}
\newtheorem{prop}[thm]{Proposition}
\theoremstyle{corollary}
\newtheorem{cor}[thm]{Corollary}
\theoremstyle{remark}
\newtheorem{example}[thm]{Example}
\theoremstyle{plain}

\theoremstyle{problem}

\theoremstyle{conclusion}

\newtheorem*{cond(V)}{Condition (V)}

\address[Hisashi Kasuya]{
Department of Mathematics, Graduate School of Science, Osaka University, Osaka,  Japan. }
\email{kasuya@math.sci.osaka-u.ac.jp}


\newcommand{\C}{\mathbb{C}}

\newcommand{\Z}{\mathbb{Z}}

\begin{document} 

\maketitle
\begin{abstract}
We give a complex polarized variation of Hodge structure over  a compact K\"ahler manifold  $M$ which controls all finite-dimensional  complex polarized variations of Hodge structure over $M$ and their tensor relations.
As a corollary, we obtain  the cohomology algebra with values in a local system admitting multiplicative Hodge structures.

\end{abstract}

\section{Introduction}
The purpose of this paper is to study  structures of the  category $\mathcal{VHS}_{\C}(M)$ of (finite-dimensional) complex polarized variations of Hodge structure over a compact K\"ahler manifold $M$ related to representation theory and cohomology theory.
By using non-abelian Hodge correspondence,  in \cite{Sim},  Simpson shows that objects in $\mathcal{VHS}_{\C}(M)$ are characterized by following objects:
\begin{description}
\item[(S1)] Semi-simple flat vector bundles (=poly-stable Higgs bundles with vanishing Chern classes) which are invariant under  canonical $\C^{\ast}$-deformations (\cite[Section 4]{Sim}).
\item[(S2)] Complex Hodge representations of the  reductive completion of the fundamental group $\pi_{1}(M,x)$ equipped with the   pure non-abelian Hodge structure (\cite[Section 5]{Sim}).

\end{description}
In this paper, we give more  discussions for connecting these arguments to standard Hodge theory (complex Hodge structures).
By using the complex Hodge structure on the cohomology and Simpson's characterization (S1), we show that any complex variation of Hodge structure on a compact K\"ahler manifold  can be simplified  as the following.
\begin{prop}[Proposition \ref{semva}]
Any $(E,D, F^{\ast},  G^{\ast}) \in {\rm Obj}(\mathcal{VHS}_{\C}(M))$ is isomorphic to
\[\bigoplus_{V\in {\mathcal V}^{s}_{VHS_{\C}}(M)} {\mathcal H}^{0}(M, {\rm Hom}(V,E))\otimes V
\]
with the natural filtrations where ${\mathcal V}^{s}_{VHS_{\C}}(M)$ is the set of  isomorphism classes of simple flat bundles admitting  (unique) structures of complex polarized variations of Hodge structure.
\end{prop}

 By using this simplification and the theory  of Tannakian category, we give a reformulation of Simpson's characterization (S2) for studying  $\mathcal{VHS}_{\C}(M)$ in terms of complex Hodge structures.
 \begin{thm}[Theorem \ref{EQHo}]
 There exists a canonical pro-algebraic group  $\varpi_{1}^{VHS_{\C}}(M,x)$ over $\C$ whose   Hopf algebra ${\mathcal O}(\varpi_{1}^{VHS_{\C}}(M,x))$  of global sections of  the structure sheaf admits a canonical complex  Hodge structure of weight $0$ depending on $x\in M$ such that there exists an equivalence between the category ${\rm Rep}^{HS}(\varpi_{1}^{VHS_{\C}}(M,x))$ of Hodge representations of  $\varpi_{1}^{VHS_{\C}}(M,x)$ and the category  $\mathcal{VHS}_{\C}(M)$.
 \end{thm}
 
The main ingredient of this paper is to construct a infinite-dimensional  complex polarized variation of Hodge structure ${\mathfrak O}^{VHS_{\C}}(x)$  depending on a point $x\in M$ satisfying certain universal properties.
 We can see   every complex polarized variation of Hodge structure over a compact K\"ahler manifold $M$ can be presented by ${\mathfrak O}^{VHS_{\C}}(x)$ (Proposition \ref{modul}).
 Moreover, ${\mathfrak O}^{VHS_{\C}}(x)$ contains information not only on each object in $\mathcal{VHS}_{\C}(M)$ but also on  their tensor relations.
This point is very important for considering the de Rham complex  with values in  ${\mathfrak O}^{VHS_{\C}}(x)$.
We note that the usual de Rham complex on a compact K\"ahler manifold admits a canonical  structure of   a bidifferential bigraded algebra giving  the  cohomology algebra with multiplicative Hodge structures.
On the other hand, the de Rham complex with values in a complex polarized variation of Hodge structure admits a canonical structure of  a double complex giving  the  cohomology  with  Hodge structures but not a algebra structure.
We have:
\begin{thm}[Theorem \ref{deRH}]
Consider the de Rham  complex 
\[A^{\ast}(M, {\mathfrak O}^{VHS_{\C}}(x))\]
with values in the local system $ {\mathfrak O}^{VHS_{\C}}(x)$.
Then:
\begin{enumerate}
\item The cochain complex  $A^{\ast}(M, {\mathfrak O}^{VHS_{\C}}(x))$ admits a canonical  differential graded algebra structure.
\item The differential graded algebra   $A^{\ast}(M, {\mathfrak O}^{VHS_{\C}}(x))$ admits a canonical  bidifferential  bigraded algebra structure.
\item $A^{\ast}(M, {\mathfrak O}^{VHS_{\C}}(x))$ is a rational $\varpi_{1}^{VHS_{\C}}(M,x)$-module such that  $\varpi_{1}^{VHS_{\C}}(M,x)$ acts as automorphisms of the  differential graded algebra.
\item The co-module structure $A^{\ast}(M, {\mathfrak O}^{VHS_{\C}}(x))\to  A^{\ast}(M, {\mathfrak O}^{VHS_{\C}}(x))\otimes {\mathcal O}(\varpi_{1}^{VHS_{\C}}(M,x))$ is a morphism of bidifferential  bigraded algebras.
\item For $(U,F^{\ast}, G^{\ast})\in {\rm Obj}({\rm Rep}^{HS}(\varpi_{1}^{VHS_{\C}}(M,x)))$, 
\[(A^{\ast}(M, {\mathfrak O}^{VHS_{\C}}(x))\otimes U)^{\varpi_{1}^{VHS_{\C}}(M,x)}
\]
is a double complex.
\item For $( E, D, F^{\ast},  G^{\ast})\in {\rm Obj}(\mathcal{ VHS}_{\C}(M))$, the double complex $(A^{\ast}(M, {\mathfrak O}^{VHS_{\C}}(x))\otimes E_{x})^{\varpi_{1}^{VHS_{\C}}(M,x)}$ is isomorphic to the canonical double complex 
\[(A^{\ast}(M,E)^{p,q}, D^{\prime},D^{\prime\prime})\]
associated with $( E, D, F^{\ast},  G^{\ast})\in {\rm Obj}(\mathcal{ VHS}_{\C}(M))$.
\end{enumerate}

\end{thm}
By this, we obtain the cohomology algebra $H^{\ast}(M, {\mathfrak O}^{VHS_{\C}}(x))$ with the multiplicative complex Hodge structures (Corollary \ref{CCHo}).

\section{Standard Complex Hodge theory}
General references of this sections are \cite{Del, PS}.
A complex Hodge structure ($\C$-HS)
of weight $w\in \Z$ is $(V,F^{\ast}, G^{\ast})$ so that 
\begin{itemize}
\item $V$ is a finite-dimensional vector space;
\item $F^{\ast}$ and $G^{\ast}$ are decreasing filtrations of $V$ so that $V = \bigoplus_{p+q=w} V^{p,q}$ where $V^{p,q}=F^{p}(V) \cap G^{q}(V)$.
\end{itemize}
Conversely, a bigrading $V = \bigoplus_{p+q=w} V^{p,q}$ of a complex vector space determines a $\C$-HS $(V,F^{\ast}, G^{\ast})$ of weight $w$ so that $F^{r}(V)= \bigoplus_{r\ge p} V^{p,q}$ and $G^{r}(V)= \bigoplus_{r\ge q} V^{p,q}$.
 A hermitian form $h$ on $V$ is a polarization of  the Hodge structure
if $V = \bigoplus_{p+q=w} V^{p,q}$ is $h$-orthonormal and $h$ is positive on $V^{p,q}$ for even $p$  and negative for  odd $p$.
A morphism of $\C$-HSs is a $\C$-linear map which is compatible with  filtrations equivalently a $\C$-linear map which is compatible with bigradings.
We can say that any morphism of $\C$-HSs is strictly compatible with filtrations and we can say that the category of complex Hodge structures of weight $w$ is abelian.

Obviously, for all integers $w\in \Z$, the categories  of complex Hodge structures of weight $w$ are equivalent.
Hence we mainly consider the abelian category  $\mathcal{HS}_{\C}$ of  complex  Hodge structures of weight $0$. 
We notice that $\mathcal{HS}_{\C}$ is equivalent to the category ${\rm Rep}(\C^{\ast})$ of finite-dimensional rational representations of the algebraic torus $\C^{\ast}$ by the following correspondence.
For $(V,\rho)\in {\rm Rep}(\C^{\ast})$, we take the bigrading  $V = \bigoplus V^{p,-p}$ so that $V^{p,-p}=\{v\in V\vert\forall t\in \C^{\ast},  \rho(t)v=t^{p}v \}$.

More generally, we also say that a direct sum of objects in $\mathcal{HS}_{\C}$ is a $\C$-HS.
This corresponds to a possibly infinite-dimensional  rational $\C^{\ast}$-module.

\section{The category of complex variations of Hodge structure}
\subsection{Higgs bundles}
Let  $M$ be  a compact K\"ahler manifold.
A Higgs bundle over $M$ is a pair
$(E, \theta)$ consisting of a  holomorphic vector bundle $E$ over $M$ and a section
$\theta\in A^{1,0}(M,{\rm 
End}(E))$ satisfying the following two conditions:
$$\bar\partial\theta =0\ \ \text{ and }
\ \ \theta\wedge \theta=0\, .$$
This section $\theta$ is called a Higgs field on $E$. 
We assume that the all Chern classes  of $E$ vanish.

Let $H$ be a Hermitian metric on $E$. 
Define $\bar\theta_{H}\,\in\, A^{0,1}(M,{\rm End}(E))$ by
$(\theta (e_{1}),\, e_{2})\,=\,(e_{1}, \bar\theta_{H} (e_{2}))$
for $e_{1},\,e_{2}\in E$.
Let $\nabla$ be the canonical  unitary connection on $E$ associated to $H$ 
Define the connection $D=\nabla+\theta+\bar\theta_{H}$ and consider the
curvature $R^{D}=D^{2}$ of $D$.
We say that $H$ is {\em harmonic} if $R^{D}=0$ (i.e. $(E,D)$ is a flat bundle).

We say that  $(E, \theta)$  is {\em stable} if  $E$   for every 
sub-Higgs sheaf ${\mathcal V}$ of  $0<{\rm rk} (\mathcal V)\,<\,{\rm 
rk}(E)$, 
the inequality
${\rm deg}(\mathcal V)<0$
holds.

\begin{thm}[\cite{Sim1}]\label{sim1}
If $(E, \theta)$ is stable, then $(E,\theta)$ admits a harmonic metric.
\end{thm}

$(E, \theta)$ is called {\em polystable} if
$
(E, \theta)\,=\, \bigoplus_{i=1}^k (E_i, \theta_i)\, ,
$
where each $(E_i, \theta_i)$ is a stable Higgs bundle.
By Theorem \ref{sim1}, if $(E, \theta)$ is polystable, then $(E, \theta)$ admits a harmonic metric.
Thus we correspond polystable $(E, \theta)$ to a flat bundle $(E,D)$.

\begin{thm}[\cite{Sim}]\label{simp}
The correspondence $(E, \theta)\mapsto (E,D)$ via harmonic metrics is an equivalence between the category of stable (resp. polystable) Higgs bundles with vanishing Chern classes and the category of simple (resp. semi-simple) flat complex vector bundles.
\end{thm}

\subsection{Complex Variations of Hodge structure}\label{CVHSssec}
 Let $M$ be a complex  manifold.
 A {\em  complex variation of Hodge structure}  ($\C$-VHS) of weight $w$ over  $M$ is
  $( E, D, F^{\ast},  G^{\ast})$ so that:
\begin{enumerate}
\item $ E$ is a flat complex vector bundle with a flat connection $D$.
\item  $ F^{\ast}$ and $ G^{\ast}$ are decreasing filtration of  $ E$ satisfying
the Griffiths transversality conditions $DF^{r}\subset A^{1}(M,F^{r-1})$ and $D G^{r}\subset A^{1}(M,  G^{r-1})$.
\item  For any $x\in M$, the fiber $(E_{x},   F^{\ast}_{x}, G^{\ast}_{x})$ at $x$ is a $\C$-HS.
\end{enumerate}
Equivalently  a   $\C$-VHS of weight $w$ over  $M$ 
 is $ (E=\bigoplus_{p+q=w} E^{p,q},D)$ so that 
  \begin{enumerate}
  \item 
$ E$ is   a $C^{\infty}$-complex vector bundle 
with a decomposition $\bigoplus_{p+q=w} E^{p,q}$ in a direct sum of  $C^{\infty}$-subbundles.
\item $D$ is  a flat connection  satisfying the Griffiths transversality conditions
\begin{multline*}
D: A^{0}(M,E^{p,q}) \to  A^{0,1}(M,E^{p+1,q-1}) \oplus A^{1,0}(M,E^{p,q}) \oplus A^{0,1}(M,E^{p,q}) \oplus A^{1,0}(M,E^{p-1,q+1}).
\end{multline*}
\end{enumerate}

A polarization
 $h$ of a $\C$-VHS is 
a parallel  Hermitian form  so that  the 
 decomposition  $\bigoplus_{p+q=w} E^{p,q}$  is orthogonal and   $h$ is positive on $E ^{p,q}$ for even $p$  and negative for  odd $p$.
 A $\C$-VHS is polarizable if it admits a polarization.

By the Griffiths transversality, the differential $D$ on $A^{\ast}(M, E)$  decomposes 
$D=\partial+\theta+ \bar\partial +\bar\theta$ so that:
\[\partial: A^{a,b}(M,E^{c,d})\to  A^{a+1,b}(M,E^{c,d}) ,\qquad
\bar\partial: A^{a,b}(M, E^{c,d})\to  A^{a,b+1}(M, E^{c,d}) ,
\]
\[\theta: A^{a,b}(M, E^{c,d})\to  A^{a+1,b}(M, E^{c-1,d+1}) \qquad
\ \ \text{ and }
\ \ \bar\theta: A^{a,b}(M,E^{c,d})\to  A^{a,b+1}(M, E^{c+1,d-1}).
\]

We define the bigrading 
\[A^{\ast}(M,E)^{p,q}=\bigoplus_{a+c=p,b+d=q}A^{a,b}(M,E^{c,d}),
\]
$D^{\prime}=\partial+ \bar\theta:A^{\ast}(M,E)^{p,q}\to A^{\ast}(M,E)^{p+1,q}$ and $D^{\prime\prime}=\bar\partial+\theta:A^{\ast}(M,E)^{p,q}\to A^{\ast}(M,E)^{p,q+1}$.
By the flatness $DD=0$, we have
\[D^{\prime}D^{\prime}=D^{\prime\prime}D^{\prime\prime}=D^{\prime}D^{\prime\prime}+D^{\prime\prime}D^{\prime}=0.
\]
We have the double complex 
\[(A^{\ast}(M,E)^{p,q}, D^{\prime},D^{\prime\prime})\]
 as  the usual  Dolbeault complex on a complex manifold.

\begin{thm}[\cite{Zu}]\label{FUnho}
Let $M$ be a compact K\"ahler   manifold.
For any polarized complex variation of Hodge structure of weight $w$ $(E=\bigoplus_{p+q=w}E^{p,q},D,h)$ over $M$,
the filtrations $F^{r}=\bigoplus_{r\le p}A^{\ast}(M,E)^{p,q}$ and $G^{r}=\bigoplus_{r\le q}A^{\ast}(M,E)^{p,q}$ induce a canonical complex Hodge structure of weight $i+w$ on the cohomology $H^{i}(M,E)$.
\end{thm}
We review this fundamental fact more precisely.
We define the differential operator $D^{c}=\sqrt{-1}(D^{\prime\prime}-D^{\prime})$.
By the Hermitian  metric  on $ E$ associated with the polarization $h$ and the K\"ahler metric $g$, we define the adjoints
$D^{\ast}$, $(D^{\prime})^{\ast}$, $(D^{\prime\prime})^{\ast}$
and $(D^{c})^{\ast}$ of  differential operators.
For the K\"ahler form $\omega$ associated with $g$, we  consider the adjoint operator $\Lambda$ of the Lefschetz operator $A^{\ast}(M,E)\ni \alpha\mapsto \omega\wedge \alpha \in A^{\ast+2}(M, E)$.
In the same way as   the usual K\"ahler identity, we have 
\[[\Lambda, D]=-(D^{c})^{\ast}
\]
and  this equation gives
\[\Delta_{D}=2\Delta_{D^{\prime}}=2\Delta_{D^{\prime\prime}}
\]
where $\Delta_{D}$, $\Delta_{D^{\prime}}$ and $\Delta_{D^{\prime\prime}}$  are the Laplacian operators (see \cite{Zu}).
Write
\[{\mathcal H}^{r}(M, E)={\rm ker}(\Delta_{D})_{\vert  A^{r}(M,E) } \qquad {\rm and}\ \qquad {\mathcal H}^{\ast}(M, E)^{P,Q}={\rm ker}(\Delta_{D^{\prime\prime}})_{\vert ( A^{\ast}(M, E))^{P,Q} }.\]
Then we have the Hodge  decomposition 
\[{\mathcal H}^{r}(M, E)=\bigoplus_{P+Q=n+r}{\mathcal H}^{r}(M,E)^{P,Q}.
\]
By ${\mathcal H}^{r}(M, E)\cong H^{r}(M, E)$, we obtain a bigrading of $H^{r}(M, E)$.
As the ordinary case $H^{\ast}(M,\C)$,  the complex Hodge structure given in Theorem \ref{FUnho} is identified with this bigrading.
Since $\Lambda $ is a map of degree $-2$, by the K\"ahler identity,  we have the following useful equations
\[{\mathcal H}^{0}(M, E)^{P,Q}={\rm ker} D_{\vert  A^{0}(M, E)^{P,Q} }={\rm ker} D^{\prime}_{\vert  A^{0}(M, E)^{P,Q} }={\rm ker} D^{\prime\prime}_{\vert  A^{0}(M, E)^{P,Q} }.
\]

\subsection{Higgs bundles and $\C$-VHSs}
For  a polarizable  $\C$-VHS $( E, D, F^{\ast},  G^{\ast})$,   by the decomposition $E=\bigoplus_{p+q=w} E^{p,q}$ as above,  $E$ is  regarded as a holomorphic vector bundle with the Dolbeault operator $\bar\partial$ and  $(E, \theta)$ is a Higgs bundle.
This correspondence is an example of Theorem \ref{simp}.
The harmonic metric is a Hermitian metric associated with a polarization $h$.
We define the $\C^{\ast}$-action $\phi: \C^{\ast}\to GL(E)$ so that $\phi(t)e=t^{p}e$ for $e\in E^{p,q}$.
Then   $\phi(t)$ is an isomorphism $(E,\theta)\cong (E,t\theta)$.

We can characterize polarizable $\C$-VHSs in terms of Higgs bundles.
\begin{prop}[{\rm \cite[Lemma 4.1]{Sim}}]\label{HIVH}
Let $(E, \theta)$ be a polystable Higgs bundle.
If for some $t\in \C^{\ast}$ which is not a root of unity we have an  isomorphism $f: (E,\theta)\cong (E,t\theta)$,
then $f$ defines  a polarizable $\C$-VHS $( E, D, F^{\ast},  G^{\ast})$ so that  $(E, \theta)$ is given by the decomposition $E=\bigoplus_{p+q=w} E^{p,q}$.
\end{prop}
Consider the equivalence $(E, \theta)\mapsto (E,D)$ in Theorem \ref{simp}.
For each $t\in \C^{\ast}$,  corresponding  $(E, t\theta)\mapsto (E,D_{t})$, we can define the $\C^{\ast}$-deformations $\{(E,D_{t})\}_{t\in \C^{\ast}}$ of semi-simple flat bundle $(E,D)$.

\begin{cor}\label{stHo}
Let $(E,D)$ be a simple flat bundle.
If for some $t\in \C^{\ast}$ which is not a root of unity we have a non-trivial morphism $f: (E,D)\to (E,D_{t})$,
then there exists a unique polarizable $\C$-VHS $( E, D, F^{\ast},  G^{\ast})$ of weight $0$.

\end{cor}

\subsection{Categories}
Let $M$ be a connected smooth   manifold.
Denote by $Fl(M)$ the category of the complex flat vector bundles over $M$ and by $Fl^{s}(M)$ (resp. $Fl^{ss}(M)$) the full sub-category  of $Fl(M)$ whose objects are simple (resp. semi-simple) complex flat vector bundles.
Suppose $M$ is a compact K\"ahler manifold.
By the same reason as the  $\C$-HS case, it is sufficient to consider $\C$-VHSs of weight $0$.
Define the category $\mathcal{VHS}_{\C}(M)$  so that objects are polarizable $\C$-VHSs of weight $0$ over $M$ and for $( E_{1}, D_{1}, F^{\ast}_{1},  G^{\ast}_{1}), ( E_{2}, D_{2}, F^{\ast}_{2},  G^{\ast}_{2})\in {\rm Obj}(\mathcal{VHS}_{\C}(M))$, 
\[{\rm Mor}_{\mathcal{VHS}_{\C}(M)} (( E_{1}, D_{1}, F^{\ast}_{1},  G^{\ast}_{1}), ( E_{2}, D_{2}, F^{\ast}_{2},  G^{\ast}_{2}))={\mathcal H}^{0}(M, {\rm Hom}(E_{1},E_{2}))^{0,0}.
\]

\begin{remark}\label{remmor}
Let ${\rm Rep}(\pi_{1}(M,x))$ be the category of finite-dimensional complex  representation of the fundamental group $\pi_{1}(M,x)$.
We consider the  functor $Fl(M)\ni (E,D)\mapsto (E_{x}, \rho_{D}) \in {\rm Rep}(\pi_{1}(M,x))$ given by taking the monodromy representations $\rho_{D}: \pi_{1}(M,x)\to GL(E_{x})$ of flat bundles.
It is well-known that this functor is an equivalence.
This equivalence gives  an isomorphism $H^{0}(M, E)\cong E_{x}^{\pi_{1}(M,x)}$.

For a polarizable $\C$-VHS $( E, D, F^{\ast},  G^{\ast})$ over a compact  K\"ahler manifold $M$, 
$E_{x}^{\pi_{1}(M,x)} $ is a complex sub-Hodge structure  in the fiber   $( E_{x},  F^{\ast}_{x},  G^{\ast}_{x})$.
This structure is isomorphic to the canonical  complex Hodge structure on $H^{0}(M, E)$ hence this does not depend on the choice of $x\in M$.
By the equivalence $Fl(M)\ni (E,D)\mapsto (E_{x}, \rho_{D}) \in {\rm Rep}(\pi_{1}(M,x))$, 
 we have an isomorphism  \[{\mathcal H}^{0}(M, {\rm Hom}(E_{1},E_{2}))^{0,0}\cong {\rm Hom}_{\pi_{1}(M,x)}(E_{1x},E_{2x})^{0,0}\] for $( E_{1}, D_{1}, F^{\ast}_{1},  G^{\ast}_{1}), ( E_{2}, D_{2}, F^{\ast}_{2},  G^{\ast}_{2})\in {\rm Obj}(\mathcal{VHS}_{\C}(M))$.

\end{remark}
Denote by $Fl_{VHS_{\C}}(M)$ (resp. $Fl^{s}_{VHS_{\C}}(M)$) the full sub-category of  $Fl(M)$ (resp. $Fl^{s}(M)$) whose objects are  $(E,D)\in {\rm Obj}(Fl(M))$ (resp. ${\rm Obj}(Fl^{s}(M)$) which come from  $(E,D, F^{\ast},  G^{\ast}) \in {\rm Obj}(\mathcal{VHS}_{\C}(M))$.
Then it is known that $Fl_{VHS_{\C}}(M)$ is a subcategory of $Fl^{ss}(M)$ (see \cite{Gri}).
By Corollary \ref{stHo}, we have the following statement.
\begin{prop}\label{uniq}
For each  $(E,D)\in  {\rm Obj}(Fl^{s}_{VHS_{\C}}(M))$,
$(E,D, F^{\ast},  G^{\ast}) \in {\rm Obj}(\mathcal{VHS}_{\C}(M))$ is unique up to isomorphism.
\end{prop}
Let ${\mathcal V}^{s}_{VHS_{\C}}(M)$ be the set of all  isomorphism classes in $Fl^{s}_{VHS_{\C}}(M)$.
By Proposition \ref{uniq}, this can be seen as a set of  isomorphism classes  in $\mathcal{VHS}_{\C}(M)$.

\begin{prop}\label{semva}
Any $(E,D, F^{\ast},  G^{\ast}) \in {\rm Obj}(\mathcal{VHS}_{\C}(M))$ is isomorphic to
\[\bigoplus_{V\in {\mathcal V}^{s}_{VHS_{\C}}(M)} {\mathcal H}^{0}(M, {\rm Hom}(V,E))\otimes V
\]
with the natural filtrations.
\end{prop}
\begin{proof}
We can easily say that the canonical map
\[\bigoplus_{V\in {\mathcal V}^{s}_{VHS_{\C}}(M)} {\mathcal H}^{0}(M, {\rm Hom}(V,E))\otimes V \ni\sum f_{i}\otimes v_{i} \mapsto \sum f_{i}(v_{i})\in V
\]
is injective and compatible with filtrations.

Let ${\mathcal V}^{s}(M)$ be the set of all  isomorphisms classes in $Fl^{s}(M)$.
In $Fl^{ss}(M)$, $(E,D)$  is isomorphic to
\[\bigoplus_{V\in {\mathcal V}^{s}(M)} {\mathcal H}^{0}(M, {\rm Hom}(V,E))\otimes V.
\]
Hence, it is sufficient to prove the following lemma.

\end{proof}
\begin{lemma}\label{Lemsi}
For $V\in {\rm Obj}(Fl^{s}(M))$, if ${\mathcal H}^{0}(M, {\rm Hom}(V,E))\not=0$,  then $V \in  {\rm Obj}(Fl^{s}_{VHS_{\C}}(M))$.
\end{lemma}

\begin{proof}
Let $E_{1}={\mathcal H}^{0}(M, {\rm Hom}(V,E))\otimes V$.
Then we have a decomposition  $E=E_{1}\oplus E_{2}$ of a flat bundle $(E,D)$.
Consider the $\C^{\ast}$-action $\phi: \C^{\ast}\to GL(E)$ associated with $E=\bigoplus_{p+q=w} E^{p,q}$.
 For $t\in \C^{\ast}$, define the morphism $f_{t}:E_{1}\to E_{1}$ by the composition of the injection $E_{1}\hookrightarrow E$, $\phi(t)$ and the projection $E\to E_{1}$.
By $f_{1}={\rm id} $, for $t\in \C^{\ast}$ sufficiently  close to $1$, $f_{t}$ is non-trivial.
Lemma \ref{Lemsi} follows from Corollary \ref{stHo}.
\end{proof}

\section{The complex variations of Hodge structures with universal properties}
\subsection{Tannakian Categories and Hodge representations}
A category $\mathcal C$ with a functor $\otimes: {\mathcal C}\times {\mathcal C}\to {\mathcal C}$  is a {\em (additive)  ${\mathbb K}$-tensor category} if:
\begin{itemize}
\item  $\mathcal C$ is an additive $\C$-linear category.
\item $\otimes : {\mathcal C}\times {\mathcal C}\to {\mathcal C}$ is a bi-linear functor which satisfies the associativity and commutativity. (see \cite{DM})
\item There exists an identity object $({\bf 1}, u)$ that is ${\bf 1}\in {\rm Ob}({\mathcal C})$ with an isomorphism $u:{\bf 1}\to {\bf 1}\otimes {\bf 1}$ satisfying the functor ${\mathcal C}\ni V\mapsto {\bf 1}\otimes V\in {\mathcal C}$ is an equivalence of categories.

\end{itemize}
A $\C$-tensor category $\mathcal C$ is {\em rigid}  if all objects admit duals.

For two $\C$-tensor categories   $({\mathcal C}_{1},\otimes_{1})$ and $({\mathcal C}_{2},\otimes_{2})$,
a {\em tensor functor} is a functor $F:{\mathcal C}_{1}\to {\mathcal C}_{2}$ with a functorial isomorphism $c_{U,V}:F(U)\otimes F(V)\to F(U\otimes V)$ so that $(F,c)$ is compatible with the associativities and commutativities of $({\mathcal C}_{1},\otimes_{1})$ and $({\mathcal C}_{2},\otimes_{2})$ (see \cite{DM}) and 
for an identity object  $({\bf 1}, u)$ of $({\mathcal C}_{1},\otimes_{1})$ $(F({\bf 1}), F(u))$ is an identity object of $({\mathcal C}_{2},\otimes_{2})$.

The category ${\rm Vect}_{\C}$ of $\C$-vector space with the usual tensor product $\otimes$ is a tensor category.
For a tensor category $({\mathcal C},\otimes)$, an exact faithful tensor functor ${\mathcal C}\to {\rm Vect}_{\C}$ is called a {\em fiber functor} for ${\mathcal C}$.
A {\em neutral $\C$-Tannakian category}  is an abelian rigid   $\C$-tensor category $\mathcal C$ with a fiber functor $\omega: {\mathcal C}\to {\rm Vect}_{\C}$  such that   $\C= {\rm End}(\bf 1)$.
\begin{example}
Let $G$ be a  pro-algebraic group $G$ over $\C$.
Let ${\rm Rep}(G)$ be the of finite-dimensional rational  representations $(V,\rho)$ of the pro-algebraic group $G$.
Then ${\rm Rep}(G)$ equipped with the natural fiber functor $(V,\rho) \mapsto V$ is  neutral $\C$-Tannakian category.

\end{example}

\begin{thm}[\cite{DM}]\label{Tandual}
Every neutral  $\C$-Tannakian category $\mathcal C$ is equivalent to the category ${\rm Rep}(G)$ of finite-dimensional rational representations of a pro-algebraic group $G$ over $\C$.
More precisely,  this correspondence $\mathcal C\mapsto G$ is  a contravariant  functor $\Pi$ from the category of neutral  Tannakian categories to the category of pro-algebraic groups over $\C$
where morphisms of neutral  Tannakian categories are exact faithful tensor functors commuting with fiber functors.
\end{thm}
We call $G=\Pi(\mathcal C)$ the Tannakian dual of a neutral  $\C$-Tannakian category $\mathcal C$.

\begin{example}
Consider the  category ${\mathcal HS}_{\C}$ of  complex  Hodge structures of weight $0$.
 We can define the tensor product of  $\C$-HSs.
 Thus, $\mathcal{ HS}_{\C}$ is a $\C$-tensor category.
Since $\mathcal{ HS}_{\C}$ is abelian, we can easily say that ${\mathcal HS}_{\C}$ is a neutral $\C$-Tannakian category with the fiber functor ${\rm Ob}(\mathcal{ HS}_{\C})\ni (V, F^{\ast})\mapsto V\in  {\rm Ob}({\rm Vect}_{\C})$.
The Tannakian dual of $\mathcal{ HS}_{\C}$ is the algebraic torus $\C^{\ast}$.

\end{example}

\begin{example}
Let $\Gamma$ a group.
Let ${\rm Rep}(\Gamma)$ be the of finite-dimensional complex  representations $(V,\rho)$ of the  group $\Gamma$.
Then ${\rm Rep}(\Gamma)$ equipped with the natural fiber functor $(V,\rho) \mapsto V$ is  neutral $\C$-Tannakian category.
It is well known that the Tannakian dual of ${\rm Rep}(\Gamma)$ is the pro-algebraic completion of $\Gamma$ defined by the inverse limit of the system of representations $\rho: \Gamma\to  G$ with the Zariski-dense images for complex  algebraic groups $G$ (see \cite{Sim}).
\end{example}

Let $M$ be a connected smooth   manifold and $x\in M$.
Then the category  $Fl(M)$ is  a neutral  $\C$-Tannakian category with the fiber functor ${\rm Ob}(Fl(M))\ni (E,D)\mapsto E_{x}\in  {\rm Ob}({\rm Vect}_{\C})$.
We denote by $\varpi_{1}(M,x)$  the Tannakian dual of $Fl(M)$.
 $Fl^{ss}(M)$ is  a semi-simple neutral  $\C$-Tannakian category.
 We denote by $\varpi_{1}^{red}(M,x)$  the Tannakian dual of $Fl^{ss}(M)$.
 Then, $\varpi_{1}^{red}(M,x)$ is  pro-reductive. 
 By the monodromy functor $Fl(M)\to {\rm Rep}(\pi_{1}(M,x))$, 
  $\varpi_{1}(M,x)$ is identified with  the algebraic completion   of the fundamental group $\pi_{1}(M,x)$ and $\varpi_{1}^{red}(M,x)= \varpi_{1}(M,x)/R_{u}(\varpi_{1}(M,x))$ where $R_{u}(\varpi_{1}(M,x))$ is the pro-unipotent radical of $\varpi_{1}(M,x)$  the natural functor $Fl^{ss}(M)\to Fl(M)$ corresponds to the quotient  map.

Suppose $M$ is a compact K\"ahler manifold.
Then $Fl_{VHS_{\C}}(M)$ is a neutral  $\C$-Tannakian category.
We denote by $\varpi_{1}^{VHS_{\C}}(M,x)$  the Tannakian dual of $Fl_{VHS_{\C}}(M)$.
Corresponding to the natural functor $Fl_{VHS_{\C}}(M)\to Fl^{ss}(M)$, we have the surjection $\varpi_{1}^{red}(M,x)\to \varpi_{1}^{VHS_{\C}}(M,x)$.
Denote by ${\mathcal O}(\varpi_{1}^{VHS_{\C}}(M,x))$  the Hopf algebra of global sections of  the structure sheaf.
By the left $\varpi_{1}^{VHS_{\C}}(M,x)$-action, ${\mathcal O}(\varpi_{1}^{VHS_{\C}}(M,x))$  is a rational $\varpi_{1}^{VHS_{\C}}(M,x)$-module.
Any rational $\varpi_{1}^{VHS_{\C}}(M,x)$-module $U$ is identified with a  co-module structure $U\to U\otimes {\mathcal O}(\varpi_{1}^{VHS_{\C}}(M,x))$ (see \cite[Proposition 2.2]{DM}).
By the fundamental arguments, regarding $(U\otimes {\mathcal O}(\varpi_{1}^{VHS_{\C}}(M,x)))^{\varpi_{1}^{VHS_{\C}}(M,x)}$ as a rational $\varpi_{1}^{VHS_{\C}}(M,x)$-module via the right action 
we have an isomorphism 
\[(U\otimes {\mathcal O}(\varpi_{1}^{VHS_{\C}}(M,x)))^{\varpi_{1}^{VHS_{\C}}(M,x)}\cong U
\]
of rational $\varpi_{1}^{VHS_{\C}}(M,x)$-modules given by mapping $\sum u\otimes  f\in (U\otimes {\mathcal O}(\varpi_{1}^{VHS_{\C}}(M,x)))^{\varpi_{1}^{VHS_{\C}}(M,x)}$ to $\sum f(1)u\in U$ where $1$ is the unit element in the group $\varpi_{1}^{VHS_{\C}}(M,x)$ (the proof of \cite[Theorem 9.1]{Gro} is valid without any change).
By Proposition \ref{semva}, $Fl_{VHS_{\C}}(M)$ is semi-simple and  any simple object of  $Fl_{VHS_{\C}}(M)$ is in $Fl^{s}_{VHS_{\C}}(M)$.
Thus, as a rational  $\varpi_{1}^{VHS_{\C}}(M,x)\times \varpi_{1}^{VHS_{\C}}(M,x)$-module,
we have an isomorphism
\[{\mathcal O}(\varpi_{1}^{VHS_{\C}}(M,x))\cong \bigoplus_{V\in  {\mathcal V}^{s}_{VHS_{\C}}(M)} V_{x}^{\ast}\otimes V_{x}
\]
given by mapping
$\sum f\otimes v\in \bigoplus_{V\in  {\mathcal V}^{s}_{VHS_{\C}}(M)} V_{x}^{\ast}\otimes V_{x} $ to the function \[   \left(\varpi_{1}^{VHS_{\C}}(M,x)\ni g\mapsto \sum f(gv)\in \C\right)\in {\mathcal O}(\varpi_{1}^{VHS_{\C}}(M,x)).\]
This isomorphism corresponds the product on ${\mathcal O}(\varpi_{1}^{VHS_{\C}}(M,x) )$  to  the tensor product in the tensor category $Fl_{VHS_{\C}}(M)$.

Let $Hi(M)$ be the category of Higgs bundles $(E,\theta)$ over $M$ and  $Hi^{ps}(M)$ (resp $Hi^{s}(M)$) the full sub-category of poly-stable (resp. stable) Higgs bundles.
For each  $t\in \C^{\ast}$, we define the functor $ Hi(M)\to Hi(M)$ by  $(E,\theta)\mapsto (E,t\theta)$.
This functor  is  restricted to functors on  $Hi^{ps}(M)$ and $Hi^{s}(M)$.
Corresponding to the equivalence $Fl^{ss}(M)\cong Hi^{ps}(M)$ as Theorem \ref{simp}, we can define a  functor $a_{t}: {\rm Rep}(\varpi_{1}^{red}(M,x))\to  {\rm Rep}(\varpi_{1}^{red}(M,x))$.
By Proposition \ref{HIVH}, $(U,\rho)\in {\rm Rep}(\varpi_{1}^{red}(M,x))$ is $(U,\rho)\in  {\rm Rep}(\varpi_{1}^{VHS_{\C}}(M,x))$ if and only if for some  $t\in \C^{\ast}$ which is not a root of unity (resp.  every $t\in \C^{\ast}$) $a_{t}(U,\rho)$  is isomorphic to $(U,\rho)$.
For each simple representation  $(U,\rho)\in  {\rm Rep}(\varpi_{1}^{VHS_{\C}}(M,x))$, corresponding to an isomorphism $a_{t}(U,\rho)\cong (U,\rho)$ for    $t\in  \C^{\ast}$, we have a canonical  $\C$-HS on  $U$.
We can regard $(U,\rho)$ as $V_{x}$ with the monodromy representation  for $V\in  {\mathcal V}^{s}_{VHS_{\C}}(M)$.
It is easily check that  a canonical  $\C$-HS on  $U$ is identified with the $\C$-HS on $V_{x}$ associated with a unique $\C$-VHS structure on $V\in  {\mathcal V}^{s}_{VHS_{\C}}(M)$.
By the tannakian duality (Theorem \ref{Tandual}), corresponding to  the functor $a_{t}:  {\rm Rep}(\varpi_{1}^{VHS_{\C}}(M,x)) \to  {\rm Rep}(\varpi_{1}^{VHS_{\C}}(M,x))$ for each $t\in \C^{\ast}$, we have the  $\C^{\ast}$-action on the pro-algebraic group $ \varpi_{1}^{VHS_{\C}}(M,x)$.
Hence, $\C^{\ast}$ acts on ${\mathcal O}(\varpi_{1}^{VHS_{\C}}(M,x))$ as Hopf  algebra automorphisms.
We can say that this action implies a multiplicative  $\C$-HS on ${\mathcal O}(\varpi_{1}^{VHS_{\C}}(M,x))$ which is identified with the $\C$-HS on 
\[{\mathcal O}(\varpi_{1}^{VHS_{\C}}(M,x))\cong \bigoplus_{V\in  {\mathcal V}^{s}_{VHS_{\C}(M)}} V_{x}^{\ast}\otimes V_{x}
\]
given by each $\C$-VHS $V\in  {\mathcal V}^{s}_{VHS_{\C}}(M)$.
\begin{remark}
The functors $ Hi(M)\ni (E,\theta)\mapsto (E,t\theta)\in Hi(M)$ also define a group action of $\C^{\ast}$ on $\varpi_{1}^{red}(M,x)$.
Simpson's characterization (S2) in Introduction is given by this $\C^{\ast}$-action on $\varpi_{1}^{red}(M,x)$.
However, the Hopf algebra ${\mathcal O}(\varpi_{1}^{red}(M,x))$ may not be a rational $\C^{\ast}$-module and so may not define a $\C$-HS.
For example,  let  $M$ be the compact complex torus $T^{2}$.
By $\pi_{1}(M,x)\cong \Z^{2}$, the $\C^{\ast}$-action on $\varpi_{1}^{red}(M,x)$  gives permutations on  the product $\Pi \C^{\ast}$ indexed by an uncountable   set.
Hence the $\C^{\ast}$-action on ${\mathcal O}(\varpi_{1}^{red}(M,x))$ does not define a $\C$-HS.
\end{remark}

Define the category ${\rm Rep}^{HS}(\varpi_{1}^{VHS_{\C}}(M,x))$ so that 
\begin{itemize}
\item Objects are  $(U,F^{\ast}, G^{\ast})$ such that 
\begin{itemize}
\item $(U,F^{\ast}, G^{\ast})$ are in   $\mathcal{HS}_{\C}$.
\item $U$ are rational $\varpi_{1}^{VHS_{\C}}(M,x)$-modules such that the associated co-module structures $U\to  U\otimes {\mathcal O}(\varpi_{1}^{VHS_{\C}}(M,x))$ are  morphisms of $\C$-HSs.
\end{itemize}
\item For $(U_{1},F^{\ast}_{1}, G^{\ast}_{1}),(U_{2}, F^{\ast}_{2}, G^{\ast}_{2})\in {\rm Ob}(Rep^{HS}(\varpi_{1}^{VHS_{\C}}(M,x)))$, morphisms from $(U_{1},F^{\ast}_{1}, G^{\ast}_{1})$ to $(U_{2}, F^{\ast}_{2}, G^{\ast}_{2})$ are morphisms of  $\C$-HSs which are  compatible with  $\varpi_{1}^{VHS_{\C}}(M,x)$-module structures.

\end{itemize}

\begin{lemma}
For any $V\in {\mathcal V}^{s}_{VHS_{\C}}(M)$,  $V_{x}$ is an object in  ${\rm Rep}^{HS}(\varpi_{1}^{VHS_{\C}}(M,x))$.
\end{lemma}
\begin{proof}
It is sufficient to show that  the  co-module structure $V_{x}\to  V_{x}\otimes {\mathcal O}(\varpi_{1}^{VHS_{\C}}(M,x))$ is a $\C$-HS morphisms.
Let $\rho $ be the  representation of $\varpi_{1}^{VHS_{\C}}(M,x)$  associated with the $\varpi_{1}^{VHS_{\C}}(M,x)$-module $V_{x}$.
The co-module structure is given by $V_{x}\ni v\mapsto \rho()v\in V_{x}\otimes {\mathcal O}(\varpi_{1}^{VHS_{\C}}(M,x))$.
Regarding the domain  $V_{x}$ as a right $\varpi_{1}^{VHS_{\C}}(M,x)$-module and  $V_{x}$ in the codomain as a left $\varpi_{1}^{VHS_{\C}}(M,x)$-module, the map $V_{x}\ni v\mapsto \rho()v\in V_{x}\otimes {\mathcal O}(\varpi_{1}^{VHS_{\C}}(M,x)))$ is $\varpi_{1}^{VHS_{\C}}(M,x)\times \varpi_{1}^{VHS_{\C}}(M,x)$-equivariant.
By 
\[{\mathcal O}(\varpi_{1}^{VHS_{\C}}(M,x))\cong \bigoplus_{V\in  {\mathcal V}^{s}_{VHS_{\C}}(M)} V_{x}^{\ast}\otimes V_{x}.
\]
as  $\varpi_{1}^{VHS_{\C}}(M,x)\times \varpi_{1}^{VHS_{\C}}(M,x)$-modules,
we can consider the co-module structure as an element in
\[(V_{x}^{\ast}\otimes V_{x})^{\varpi_{1}^{VHS_{\C}}(M,x)}\otimes (V_{x}^{\ast}\otimes V_{x})^{\varpi_{1}^{VHS_{\C}}(M,x)}.
\]
and hence this is of type $(0,0)$ for the $\C$-HS.
Thus the lemma follows.

\end{proof}

\begin{thm}\label{EQHo}
We have an equivalence $\mathcal{ VHS}_{\C}(M)\cong {\rm Rep}^{HS}(\varpi_{1}^{VHS_{\C}}(M,x))$.
\end{thm}
\begin{proof}
For  $(E, D, F^{\ast}, G^{\ast})\in {\rm Obj}(\mathcal{VHS}_{\C}(M))$, by 
 \[E_{x}\cong \bigoplus_{V\in {\mathcal V}^{s}_{VHS_{\C}}(M)} {\mathcal H}^{0}(M, {\rm Hom}(V,E))\otimes V_{x}
\]
(Proposition \ref{semva}), $(E_{x}, F^{\ast}_{x}, G^{\ast}_{x})$
is an object in  ${\rm Rep}^{HS}(\varpi_{1}^{VHS_{\C}}(M,x))$.
By Remark \ref{remmor} , this correspondence is a fully faithful functor.
We can construct a quasi-inverse by the following way.
We correspond  $(U,F^{\ast}, G^{\ast})\in {\rm Obj}(Rep^{HS}(\varpi_{1}^{VHS_{\C}}(M,x)))$ to an object
\[\bigoplus_{V\in {\mathcal V}^{s}_{VHS_{\C}}(M)}{\rm Hom}_{\pi_{1}(M,x)}(V_{x}, U)\otimes V
\]
in $\mathcal{VHS}_{\C}(M)$.
\end{proof}

\subsection{The $\C$-VHS with universal properties}
We define the local system ${\mathfrak O}^{VHS_{\C}}(x)$ over $M$ associated with  ${\mathcal O}(\varpi_{1}^{VHS_{\C}}(M,x))$ as a $\pi_{1}(M,x)$-module via the left $\varpi_{1}^{VHS_{\C}}(M,x)$-action.
\begin{prop}
\begin{enumerate}
\item ${\mathfrak O}^{VHS_{\C}}(x)\cong \bigoplus_{V\in  {\mathcal V}^{s}_{VHS_{\C}}(M)} V^{\ast}\otimes V_{x}$.
\item ${\mathfrak O}^{VHS_{\C}}(x)$ admits a canonical  $\C$-VHS structure (the direct sum of objects in $\mathcal{VHS}_{\C}(M)$) depending on $x\in M$.
\item ${\mathfrak O}^{VHS_{\C}}(x)$ admits a structure ${\mathfrak O}^{VHS_{\C}}(x)\otimes  {\mathfrak O}^{VHS_{\C}}(x)\to  {\mathfrak O}^{VHS_{\C}}(x)$ of a local system of  $\C$-algebras.
\item The algebra   structure ${\mathfrak O}^{VHS_{\C}}(x)\otimes  {\mathfrak O}^{VHS_{\C}}(x)\to  {\mathfrak O}^{VHS_{\C}}(x)$ is a  morphism of $\C$-VHSs.

\end{enumerate}
\end{prop}
\begin{proof}
(1) is a consequence of 
${\mathfrak O}^{VHS_{\C}}_{x}\cong \bigoplus_{V\in  {\mathcal V}^{s}_{VHS_{\C}}(M)} V^{\ast}\otimes V_{x}$.

(2) follows from (1).

Since the algebra structure ${\mathcal O}(\varpi_{1}^{VHS_{\C}}(M,x))\otimes {\mathcal O}(\varpi_{1}^{VHS_{\C}}(M,x))\to {\mathcal O}(\varpi_{1}^{VHS_{\C}}(M,x))$ on  ${\mathcal O}(\varpi_{1}^{VHS_{\C}}(M,x))$ is equivariant for the left $\varpi_{1}^{VHS_{\C}}(M,x)$-action, this gives (3).

We have seen that  the algebra structure ${\mathcal O}(\varpi_{1}^{VHS_{\C}}(M,x))\otimes {\mathcal O}(\varpi_{1}^{VHS_{\C}}(M,x))\to {\mathcal O}(\varpi_{1}^{VHS_{\C}}(M,x))$ is a $\C$-HS morphism.
  Since  we have  an isomorphism  between
\[{\mathcal H}^{0}( M, Hom ({\mathfrak O}^{VHS_{\C}}_{x}\otimes  {\mathfrak O}^{VHS_{\C}}_{x},  {\mathfrak O}^{VHS_{\C}}_{x}))^{0,0}\]
and\[ Hom_{\pi_{1}(M,x)} ( {\mathcal O}(\varpi_{1}^{VHS_{\C}}(M,x))\otimes  {\mathcal O}(\varpi_{1}^{VHS_{\C}}(M,x)),  {\mathcal O}(\varpi_{1}^{VHS_{\C}}(M,x)))^{0,0}
\](Remark \ref{remmor}),
(4) follows.

\end{proof}

\begin{prop}\label{modul}
 ${\mathfrak O}^{VHS_{\C}}(x)$ admits a structure $\varpi_{1}^{VHS_{\C}}(M,x)\times {\mathfrak O}^{VHS_{\C}}(x)\to {\mathfrak O}^{VHS_{\C}}(x)$ of a local system of rational $\varpi_{1}^{VHS_{\C}}(M,x)$-modules so that:
 \begin{enumerate}
 \item For any $(U,F^{\ast}, G^{\ast})\in {\rm Obj}(Rep^{HS}(\varpi_{1}^{VHS_{\C}}(M,x)))$, we have  $({\mathfrak O}^{VHS_{\C}}(x)\otimes U)^{\varpi_{1}^{VHS_{\C}}(M,x)}\in \mathcal{ VHS}_{\C}(M)$.
 \item The correspondence $(U,F^{\ast}, G^{\ast})\mapsto ({\mathfrak O}^{VHS_{\C}}(x)\otimes U)^{\varpi_{1}^{VHS_{\C}}(M,x)}$ is an equivalence $\mathcal{ VHS}_{\C}(M)\cong {\rm Rep}^{HS}(\varpi_{1}^{VHS_{\C}}(M,x))$.
 \item For any  $( E, D, F^{\ast},  G^{\ast})\in {\rm Obj}(\mathcal{ VHS}_{\C}(M))$, 
 $({\mathfrak O}^{VHS_{\C}}(x)\otimes E_{x})^{\varpi_{1}^{VHS_{\C}}(M,x)}\in {\rm Obj}(\mathcal{ VHS}_{\C}(M))$ is isomorphic to $( E, D, F^{\ast},  G^{\ast})$.
 \end{enumerate}
\end{prop}
\begin{proof}
Since the right $\varpi_{1}^{VHS_{\C}}(M,x)$-action on ${\mathcal O}(\varpi_{1}^{VHS_{\C}}(M,x))$ commutes with the left action, the rational $\varpi_{1}^{VHS_{\C}}(M,x)$-module structure on ${\mathcal O}(\varpi_{1}^{VHS_{\C}}(M,x))$ via the right action defines  structure $\varpi_{1}^{VHS_{\C}}(M,x)\times {\mathfrak O}^{VHS_{\C}}(x)\to {\mathfrak O}^{VHS_{\C}}(x)$ of a local system of rational $\varpi_{1}^{VHS_{\C}}(M,x)$-modules.

By ${\mathfrak O}^{VHS_{\C}}(x)\cong \bigoplus_{V\in  {\mathcal V}^{s}_{VHS_{\C}}(M)} V^{\ast}\otimes V_{x}$, for a rational $\varpi_{1}^{VHS_{\C}}(M,x)$-module $U$, we have \[ ({\mathfrak O}^{VHS_{\C}}(x)\otimes U)^{\varpi_{1}^{VHS_{\C}}(M,x)}\cong 
\bigoplus_{V\in {\mathcal V}^{s}_{VHS_{\C}}(M)}{\rm Hom}_{\pi_{1}(M,x)}(V_{x}, U)\otimes V.
\]
Thus, the correspondence $(U,F^{\ast}, G^{\ast})\mapsto ({\mathfrak O}^{VHS_{\C}}(x)\otimes U)^{\varpi_{1}^{VHS_{\C}}(M,x)}$ is the  functor  giving Theorem \ref{EQHo}.
Hence, (1),  (2) and (3) are obvious.
\end{proof}

\begin{thm}\label{deRH}
Consider the de Rham  complex 
\[A^{\ast}(M, {\mathfrak O}^{VHS_{\C}}(x))\]
with values in the local system $ {\mathfrak O}^{VHS_{\C}}(x)$.
Then:
\begin{enumerate}
\item The cochain complex  $A^{\ast}(M, {\mathfrak O}^{VHS_{\C}}(x))$ admits a canonical  differential graded algebra structure.
\item The differential graded algebra   $A^{\ast}(M, {\mathfrak O}^{VHS_{\C}}(x))$ admits a canonical  bidifferential  bigraded algebra structure.
\item $A^{\ast}(M, {\mathfrak O}^{VHS_{\C}}(x))$ is a rational $\varpi_{1}^{VHS_{\C}}(M,x)$-module such that  $\varpi_{1}^{VHS_{\C}}(M,x)$ acts as automorphisms of the  differential graded algebra.
\item The co-module structure $A^{\ast}(M, {\mathfrak O}^{VHS_{\C}}(x))\to  A^{\ast}(M, {\mathfrak O}^{VHS_{\C}}(x))\otimes {\mathcal O}(\varpi_{1}^{VHS_{\C}}(M,x))$ is a morphism of bidifferential  bigraded algebras.
\item For $(U,F^{\ast}, G^{\ast})\in {\rm Obj}({\rm Rep}^{HS}(\varpi_{1}^{VHS_{\C}}(M,x)))$, 
\[(A^{\ast}(M, {\mathfrak O}^{VHS_{\C}}(x))\otimes U)^{\varpi_{1}^{VHS_{\C}}(M,x)}
\]
is a double complex.
\item For $( E, D, F^{\ast},  G^{\ast})\in {\rm Obj}(\mathcal{ VHS}_{\C}(M))$, the double complex $(A^{\ast}(M, {\mathfrak O}^{VHS_{\C}}(x))\otimes E_{x})^{\varpi_{1}^{VHS_{\C}}(M,x)}$ is isomorphic to the double complex 
\[(A^{\ast}(M,E)^{p,q}, D^{\prime},D^{\prime\prime})\]
as  Subsection \ref{CVHSssec}.
\end{enumerate}

\end{thm}

\begin{proof}
 $A^{\ast}(M, {\mathfrak O}^{VHS_{\C}}(x))$ is a differential graded algebra such that the multiplications are defined  by the wedge product and the algebra structure on  $ {\mathfrak O}^{VHS_{\C}}(x)$.
 (1) follows.
 
We define the double complex structure on $A^{\ast}(M, {\mathfrak O}^{VHS_{\C}}(x))$ defined  by the complex structure on $M$ and the  $\C$-VHS structure on ${\mathfrak O}^{VHS_{\C}}(x)$ as Subsection \ref{CVHSssec}.
Then, since ${\mathfrak O}^{VHS_{\C}}(x)\otimes  {\mathfrak O}^{VHS_{\C}}(x)\to  {\mathfrak O}^{VHS_{\C}}(x)$ is a  morphism of $\C$-VHSs,  the product on  $A^{\ast}(M, {\mathfrak O}^{VHS_{\C}}_{x})$ is compatible with the bigrading.
Hence,  we have the bidifferential  bigraded algebra structure on $A^{\ast}(M, {\mathfrak O}^{VHS_{\C}}(x))$.
(2) follows.

The  structure $\varpi_{1}^{VHS_{\C}}(M,x)\times {\mathfrak O}^{VHS_{\C}}(x)\to {\mathfrak O}^{VHS_{\C}}(x)$ of a local system of rational $\varpi_{1}^{VHS_{\C}}(M,x)$-modules defines a   rational $\varpi_{1}^{VHS_{\C}}(M,x)$-module structure on   $A^{\ast}(M, {\mathfrak O}^{VHS_{\C}}(x))$ such that the $\varpi_{1}^{VHS_{\C}}(M,x)$-action commutes with the differential on $A^{\ast}(M, {\mathfrak O}^{VHS_{\C}}(x))$.
Since the algebra structure ${\mathcal O}(\varpi_{1}^{VHS_{\C}}(M,x))\otimes {\mathcal O}(\varpi_{1}^{VHS_{\C}}(M,x))\to {\mathcal O}(\varpi_{1}^{VHS_{\C}}(M,x))$ on  ${\mathcal O}(\varpi_{1}^{VHS_{\C}}(M,x))$ is equivariant for the right $\varpi_{1}^{VHS_{\C}}(M,x)$-action, the $\varpi_{1}^{VHS_{\C}}(M,x)$-action on $A^{\ast}(M, {\mathfrak O}^{VHS_{\C}}(x))$ is compatible with the multiplication on $A^{\ast}(M, {\mathfrak O}^{VHS_{\C}}(x))$.
Thus (3) follows.

By ${\mathfrak O}^{VHS_{\C}}(x)\cong \bigoplus_{V\in  {\mathcal V}^{s}_{VHS_{\C}}(M)} V^{\ast}\otimes V_{x}$,
the $(p,q)$-component of the bigrading on $A^{\ast}(M, {\mathfrak O}^{VHS_{\C}}(x))$ is given by 
\[\bigoplus_{V\in  {\mathcal V}^{s}_{VHS_{\C}}(M)} \bigoplus_{a+b=p,c+d=q} A^{\ast}(M,V^{\ast})^{a,c}\otimes V^{b,d}_{x}.
\]
We can easily check (4).

The subspace  $(A^{\ast}(M, {\mathfrak O}^{VHS_{\C}}(x))\otimes U)^{\varpi_{1}^{VHS_{\C}}(M,x)} \subset A^{\ast}(M, {\mathfrak O}^{VHS_{\C}}(x))\otimes U
$
is the kernel of the map \[A^{\ast}(M, {\mathfrak O}^{VHS_{\C}}(x))\otimes U\ni \omega\mapsto \rho()\omega-\omega\in   A^{\ast}(M, {\mathfrak O}^{VHS_{\C}}(x))\otimes U\otimes {\mathcal O}(\varpi_{1}^{VHS_{\C}}(M,x))\]
where $ \omega\mapsto \rho()\omega$ is the co-module structure of ${\mathcal O}(\varpi_{1}^{VHS_{\C}}(M,x))$ on $A^{\ast}(M, {\mathfrak O}^{VHS_{\C}}(x))\otimes U$.
Thus, (5) follows from (4).

By Proposition \ref{semva},   we can identify  $(E,D, F^{\ast},  G^{\ast}) \in {\rm Obj}(\mathcal{VHS}_{\C}(M))$ with 
\[\bigoplus_{V\in {\mathcal V}^{s}_{VHS_{\C}}(M)} {\mathcal H}^{0}(M, {\rm Hom}(V,E))\otimes V.
\]
We have 
\[ A^{\ast}(M,E)^{p,q}=\bigoplus_{V\in  {\mathcal V}^{s}_{VHS_{\C}}(M)} \bigoplus_{a+b=p,c+d=q} A^{\ast}(M,V)^{a,c}\otimes {\mathcal H}^{0}(M, {\rm Hom}(V,E))^{c,d}.
\]
For 
\[E_{x}=\bigoplus_{V\in {\mathcal V}^{s}_{VHS_{\C}}(M)} {\mathcal H}^{0}(M, {\rm Hom}(V,E))\otimes V_{x},
\]
we can directly check (6). 

\end{proof}

By Theorem \ref{FUnho}, we have:
\begin{cor}\label{CCHo}
\begin{enumerate}
\item
The cohomology $H^{\ast}(M, {\mathfrak O}^{VHS_{\C}}(x))$ with values in the local system ${\mathfrak O}^{VHS_{\C}}(x)$  admits a canonical structure $H^{\ast}(M, {\mathfrak O}^{VHS_{\C}}(x))\otimes H^{\ast}(M, {\mathfrak O}^{VHS_{\C}}(x))\to H^{\ast}(M, {\mathfrak O}^{VHS_{\C}}(x))$ of a graded-commutative  $\C$-algebra.
\item $H^{i}(M, {\mathfrak O}^{VHS_{\C}}(x))$ admits  a canonical complex Hodge structure of weight $i$.
\item The multiplication $H^{i}(M, {\mathfrak O}^{VHS_{\C}}(x))\otimes H^{j}(M, {\mathfrak O}^{VHS_{\C}}(x))\to H^{i+j}(M, {\mathfrak O}^{VHS_{\C}}(x))$ is a  morphism of $\C$-HSs.

\item $H^{\ast}(M, {\mathfrak O}^{VHS_{\C}}(x))$ is the direct sum of   objects in ${\rm Rep}^{HS}(\varpi_{1}^{VHS_{\C}}(M,x))$.
\item For $( E, D, F^{\ast},  G^{\ast})\in {\rm Obj}(\mathcal{ VHS}_{\C}(M))$,   we have a $\C$-HS isomorphism  
\[(H^{i}(M, {\mathfrak O}^{VHS_{\C}}(x))\otimes E_{x})^{\varpi_{1}^{VHS_{\C}}(M,x)}\cong H^{i}(M, E).\]

\item
For $( E_{1}, D_{1}, F^{\ast}_{1},  G^{\ast}_{1}), ( E_{2}, D_{2}, F^{\ast}_{2},  G^{\ast}_{2})\in {\rm Obj}(\mathcal{ VHS}_{\C}(M))$,
as a  morphism of $\C$-HSs, the cup product $H^{i}(M, E_{1})\otimes H^{\ast}(M, E_{2})\to H^{j}(M, E_{1}\otimes E_{2})$ is identified with the restriction of the tensor product of the multiplication
\[H^{i}(M, {\mathfrak O}^{VHS_{\C}}(x))\otimes H^{j}(M, {\mathfrak O}^{VHS_{\C}}(x))\to H^{i+j}(M, {\mathfrak O}^{VHS_{\C}}(x))
\]
and the identity 
\[E_{1x}\otimes E_{2x}\to E_{1x}\otimes E_{2x}
\]
by isomorphisms 
\[(H^{\ast}(M, {\mathfrak O}^{VHS_{\C}}(x))\otimes E_{kx})^{\varpi_{1}^{VHS_{\C}}(M,x)}\cong H^{\ast}(M, E_{k})\]
for $k=1,2$.

\end{enumerate}

\end{cor}

\begin{remark}
For  $(E, D, F^{\ast}, G^{\ast})\in {\rm Obj}({\mathcal VHS}_{\C}(M))$, we have the complex  Hodge structure of weight $i$ on the cohomology 
$H^{i}(M, {\mathfrak O}^{VHS_{\C}}(x)\otimes E)$.
By
\[E_{x}=\bigoplus_{V\in {\mathcal V}^{s}_{VHS_{\C}}(M)} {\mathcal H}^{0}(M, {\rm Hom}(V,E))\otimes V_{x},
\]
we have an isomorphism  $H^{0}(M, {\mathfrak O}^{VHS_{\C}}_{x}\otimes E)\cong (E_{x}, F^{\ast}_{x}, G^{\ast}_{x})$ of $\C$-HSs.
\end{remark}

\end{document}